\documentclass[11pt]{amsart}
\usepackage{amsfonts,amsmath,amssymb,amsthm,cmtiup} 
\usepackage{mathrsfs}

\newtheorem*{theorem*}{Theorem}
\newtheorem*{lemma*}{Lemma}

\newtheorem*{fact*}{Fact} 

\newcommand{\zero}{{\mathbf 0}}

\theoremstyle{definition} 

\newtheorem*{remark*}{Remark}
\newtheorem{definition}{Definition}
\newtheorem*{definition*}{Definition}

\newtheorem*{example*}{Example}

\def\btk#1#2{\bigl[\bigotimes^{#1}{#2}\bigr]}

\begin{document}

\title[Extremally Disconnected Groups of Measurable Cardinality]{Extremally Disconnected Groups\\ of Measurable Cardinality} 
\author{Ol'ga Sipacheva}

\begin{abstract}
Given an arbitrary measurable cardinal $\kappa$, a nondiscrete 
Hausdorff extremally disconnected group of cardinality $\kappa$ is constructed.
\end{abstract}



\address{Department of General Topology and Geometry, Faculty of Mechanics and  Mathematics, 
M.~V.~Lomonosov Moscow State University, Leninskie Gory 1, Moscow, 199991 Russia}

\email{o-sipa@yandex.ru, osipa@gmail.com}

\maketitle

This note is concerned with Arhangel'skii's old problem of the existence in ZFC of a nondiscrete 
Hausdorff extremally disconnected group~\cite{Arhangelskii}. The nonexistence of a countable group 
with these properties in consistent with ZFC~\cite{RS}. Here we show that, for any measurable 
cardinal $\kappa$, there exists a nondiscrete Hausdorff extremally disconnected group of cardinality 
$\kappa$. We begin with basic definitions; for more details on measurable cardinals and related 
ultrafilters, see, e.g., \cite{Jech}. 

\begin{definition}
A topological space is \emph{extremally disconnected} if the closure of any open set is open 
in this space.
\end{definition}

\begin{definition}
Let $\kappa$ be an uncountable cardinal. A filter $\mathscr F$ is \emph{$\kappa$-complete} if 
$\bigcap_{\alpha<\lambda} A_\alpha\in \mathscr F$ for any cardinal $\lambda<\kappa$ and any 
$A_\alpha\in \mathscr F$, $\alpha <\lambda$. 
\end{definition}

\begin{definition}
An uncountable cardinal $\kappa$ is said to be \emph{measurable} if there exists a $\kappa$-complete 
nonprincipal ultrafilter on $\kappa$. 
\end{definition}

The nonexistence of measurable cardinals is consistent with ZFC, while the consistency of their 
existence cannot be proved within ZFC (see \cite[Lemma~10.4 and Theorem~12.12]{Jech}).

\begin{definition}
Let $\delta$ be an ordinal, and let $X_\alpha\subset \delta$ for $\alpha <\delta$. The  
\emph{diagonal intersection} $\Delta_{\alpha <\delta}$ of the sequence 
$(X_\alpha)_{\alpha<\delta}$ is $\{\alpha<\delta: \alpha\in \bigcap_{\beta 
<\alpha}X_\alpha$.
\end{definition}

\begin{definition}
A filter $\mathscr F$ on a cardinal $\kappa$ is \emph{normal} if it is closed under diagonal
intersections.
\end{definition}

Obviously, a  nonprincipal normal ultrafilter on $\kappa$ is $\kappa$-complete if and 
only if it is \emph{uniform} (that is, contains no elements of cardinality less than $\kappa$). 

\begin{fact*}[{see \cite[Theorem 10.20]{Jech}}]
\label{normal}
Every measurable cardinal $\kappa$ carries a normal $\kappa$-complete nonprincipal ultrafilter.
\end{fact*}

Given a set $X$ and a cardinal $\kappa$, we use the standard notation
$$
[X]^\kappa=\{Y\subset X: |Y|=\kappa\}, \qquad  [X]^{<\kappa}=\{Y\subset X: |Y|<\kappa\}.
$$

For any cardinal $\kappa$, the set $[\kappa]^{<\omega}$ 
of all finite subsets of $\kappa$ is the Boolean group freely generated by $\kappa$ with 
respect to the operation $\triangle$ of symmetric difference. For any filter $\mathscr F$ on 
$\kappa$, the subgroups $\langle A\rangle$ generated by $A\in \mathscr F$ form a base of 
neighborhoods of zero in a group topology on $[\kappa]^{<\omega}$. We denote the group 
$[\kappa]^{<\omega}$ with this topology by $B(\kappa_\mathscr F)$ and the zero element of this group 
by $\zero$. Obviously, if the filter $\mathscr F$ is free, then $B(\kappa_\mathscr F)$ is a 
nondiscrete Hausdorff group. 

Our purpose is to prove the following assertion. 

\begin{theorem*}
Given any normal $\kappa$-complete nonprincipal ultrafilter $\mathscr U$ on a cardinal $\kappa$, 
$B(\kappa_\mathscr U)$ is a nondiscrete Hausdorff extremally disconnected topological group. 
Moreover, all subsets of cardinality less than $\kappa$ in $B(\kappa_\mathscr U)$ are closed and 
discrete, and $B(\kappa_\mathscr U)$ is a $P_\kappa$-space, i.e., the intersection of any family of 
fewer than $\kappa$ open sets in $B(\kappa_\mathscr U)$ is open.
\end{theorem*}

The proof of this theorem is based on a lemma about symmetric products of ultrafilters, which 
are defined by analogy with the usual (Fubini, or tensor) products of ultrafilters  as follows. 

\begin{definition}
Let $\mathscr F$ and $\mathscr G$ be two filters on sets $X$ and $Y$, 
respectively. The family 
$$
\mathscr F\otimes \mathscr G=
\{A\subset X\times X: \{x\in X:\{y\in Y: (x,y)\in A\}\in \mathscr G\}\in \mathscr F\} 
$$
is called the \emph{product} of $\mathscr F$ and $\mathscr G$.\footnote{In the literature, 
the terms \emph{tensor product} and  
\emph{Fubini product}  and notations $\mathscr F\cdot \mathscr G$ and $\mathscr F\times 
\mathscr G$ are also used.} 
\end{definition}

\begin{definition}
Let $k$ be a positive integer, and let $\mathscr F_1$, \dots, $\mathscr F_k$ be filters on 
 a cardinal $\kappa$. We define the \emph{symmetric product} $[\mathscr F_1\otimes \dots\otimes 
\mathscr F_k$ of $\mathscr F_1]$, \dots, $\mathscr F_k$ recursively. 
The symmetric product of a single factor is 
set equal to this factor, and for $k>1$, 
\begin{multline*}
[\mathscr F_1\otimes \dots\otimes 
\mathscr F_k]\\
= \{A\subset [\kappa]^k: 
\{F\in [\kappa]^{k-1}:\{\alpha\in \kappa\setminus (\max F+1): F\cup \{\alpha\}\in 
A\}\in \mathscr F_n\}\\
\in [\mathscr F_1\otimes \dots\otimes 
\mathscr F_{k-1}]\}.
\end{multline*}
In particular, for two filters $\mathscr F$ and $\mathscr G$ on $\kappa$, 
$$
[\mathscr F\otimes \mathscr G]=
\{A\subset [\kappa]^2: 
\{\alpha\in \kappa:\{\beta \in \kappa\setminus \alpha: \{\alpha, \beta\}\in 
A\}\in \mathscr G\}\in \mathscr F\}.
$$ 

For the symmetric product of $k$ copies of a filter $\mathscr 
F$, we use the notation $\btk k{\mathscr F}$.  
\end{definition}

It is well known that the product of any two filters (ultrafilters) is 
a filter (ultrafilter); see, e.g., \cite[p.~156]{Comfort-Negrepontis}. It easily follows by 
induction that the symmetric product of any $k$ uniform ultrafilters on an infinite cardinal 
$\kappa$ is an ultrafilter (it suffices to note that the diagonal 
$\Delta=\{(\alpha,\dots,\alpha):\alpha\in \kappa\}$ of $\kappa^k$ 
is not in $[\mathscr F_1\otimes\dots\otimes \mathscr F_k]$ and none of the initial intervals of 
$\kappa$ is not in $\mathscr F_k$ and consider the direct image \cite[p.~155]{Comfort-Negrepontis} 
of the ultrafilter $[\mathscr F_1\otimes\dots\otimes \mathscr F_{k-1}]\otimes \mathscr 
F_k$ restricted to 
$$
X=\{(F, \alpha): F\in [\kappa]^{k-1}, \ \alpha\in \kappa \setminus (\max F+1)\}\subset  
[\kappa]^{k-1}\times \kappa
$$ 
under the map 
$f\colon X \to [\kappa]^k$ defined by $f((F,\alpha))=F\cup\{\alpha\}$).  
Note also that $[\mathscr F\otimes \mathscr G]$ is nonprincipal if so is $\mathscr F$ or 
$\mathscr G$. 

\begin{lemma*}
For any positive integer $k$, any cardinal $\kappa$, and any normal $\kappa$-complete ultrafilter 
$\mathscr U$ on $\kappa$, the sets $[A]^k$, $A\in \mathscr U$, form a base of the ultrafilter 
$\btk k{\mathscr U}$. 
\end{lemma*}

\begin{proof}
We prove the lemma by induction on $k$. For $k=1$, there is nothing to prove. Suppose that $k>1$ and 
the assertion holds for all smaller $k$. 

We identify each $F\in [\kappa]^{k-1}$ with the increasing $(k-1)$-tuple of the elements of $F$. 
Being endowed with the corresponding colexicographic order 
$\preccurlyeq$, $[\kappa]^{k-1}$ is a well-ordered set of order type $\kappa$ (because each  $F\in 
[\kappa]^{k-1}$ has less than $\kappa$ predecessors). 
For each $\alpha<\kappa$ and every positive 
integer $m$, let $\gamma_m(\alpha)$ denote the order type of $([\alpha+1]^m,\preccurlyeq)$. Then 
$\kappa\setminus (\gamma_{k-1}(\alpha)+1)\in \mathscr U$ 
(because $\mathscr U$ is $\kappa$-complete and, therefore, uniform), whence 
$$ 
D=\Delta_{\alpha\in 
\kappa}(\kappa\setminus (\gamma_{k-1}(\alpha)+1)\in \mathscr U \quad\text{and}\quad \alpha > \gamma_{k-1}(\beta) \text{ for 
any }\alpha, \beta \in D,\ \alpha>\beta. \eqno(\star) 
$$

Take any $\widetilde A\in \btk k {\mathscr U}$. By the induction hypothesis (and by the definition 
of the symmetric product of filters), there exists an $A'\in \mathscr U$ and a family $\{A_F: F\in 
[A']^{k-1}\}$, where $A_F\in \mathscr U$, such that $A'\subset D$ and  $\widetilde D=\{F\cup 
\{\alpha\}: F\in [A']^{k-1}, \alpha\in A_F\setminus (\max F+1)\}\subset \widetilde A$. Let us number 
the elements of $[A']^{k-1}$ by ordinals in $\preccurlyeq$-increasing order: 
$$
[A']^{k-1}=\{F_\alpha:\alpha<\kappa\},\qquad F_\alpha\prec F_\beta\quad\text{for}\quad 
\alpha<\beta. 
$$
Clearly, the ordinal number $\alpha$ of any $F_\alpha \in ([A']^{k-1},\preccurlyeq)$ does not 
exceed the ordinal number of $F_\alpha$ in $([\kappa]^k,\preccurlyeq)$. We set 
$A_\alpha=A_{F_\alpha}\setminus (\max F_\alpha + 1)$ and $A=\Delta_{\alpha\in \kappa} A_\alpha\cap 
A'$. If $\alpha_1<\dots<\alpha_k$, $\alpha_i\in A$, then $\{\alpha_1, \dots,\alpha_{k-1}\}=F_\alpha$ 
for some $\alpha\in \kappa$. Since $\alpha_k\in A\subset D$, we have $\alpha_k>\gamma_{k-1} 
(\alpha_{k-1})$. Therefore, the ordinal number of the set $F_\alpha$ in $([\kappa]^{k-1}, 
\preccurlyeq)$ is less than $\alpha_k$ (because $F_\alpha\in [\alpha_{k-1}+1]^{k-1}$), and hence 
$\alpha<\alpha_k$. Thus, it follows from $\alpha_k\in \Delta_{\beta\in \kappa}A_\beta$ that 
$\alpha_k\in A_\alpha\subset A_{F_\alpha}$. By the definition of $A_{F_\alpha}$, we have 
$\{\alpha_1, \dots, \alpha)_k\}\in \widetilde A$, and the arbitrariness of 
$\alpha_1,\dots,\alpha_k\in A$ implies $[A]^k\subset \widetilde A$. 
\end{proof}

Now we can prove the theorem.

\begin{proof}[Proof of the theorem]
Let $U\subset B(\kappa_{\mathscr U})$ be an open set such that $\zero\in \overline{U}\setminus U$. 
To prove the extremal disconnectedness of $B(\kappa_{\mathscr U})$, we must find an $A\in \mathscr 
U$ for which $\langle A\rangle \subset \overline U$. 

Since $\mathscr U$ is $\sigma$-complete, it follows that there exists a positive integer $k$ for 
which $\zero\in \overline {U\cap [\kappa]^k}$ (otherwise, for each $k$, there is an $A_k\in \mathscr 
U$ such that $\langle A_k\rangle \cap U\cap [\kappa]^k$, and for $A=\bigcap A_k$, we 
have $A\in \mathscr U$ and $\langle A\rangle \cap U=\varnothing$, which contradicts the assumption 
$\zero\in \overline{U}$). 

Let $D=\Delta_{\alpha\in 
\kappa}(\kappa\setminus (\gamma_{k}(\alpha)+1)$ (recall that $\gamma_k(\alpha)$ is the order 
type of $([\alpha+1]^{k},\preccurlyeq)$). Then $D\in \mathscr U$, and 
according to the lemma, there exists a $B\in \mathscr U$, $B\subset D$, for 
which $[B]^k\subset U\cap [\kappa]^k$ (otherwise, $U\cap [\kappa]^k\notin \btk k{\mathscr U}$ and 
there exists a $B\in \mathscr U$ such that $[B]^k\cap U\cap [\kappa]^k=\varnothing$ and hence 
$\langle B\rangle \cap U\cap [\kappa]^k$, which contradicts the assumption $\zero\in \overline 
{U\cap [\kappa]^k}$). 

Since $U$ is open, it follows that each $F\in [B]^k$ is contained in $U$ 
together with its neighborhood; in other words, for each $F\in [B]^k$, there exists an $A_F\in 
\mathscr U$ such that $F+\langle A_F\rangle=\{F\triangle G:G\in \langle A_F\rangle \}\subset U$. 
As in the proof of the 
lemma, we number the elements of $[B]^k$ by ordinals in $\preccurlyeq$-increasing order: 
$$ 
[B]^{k}=\{F_\alpha:\alpha<\kappa\},\qquad F_\alpha\prec F_\beta\quad\text{for}\quad \alpha<\beta; 
$$ 
then we set $A_\alpha=A_{F_\alpha}\setminus (\max F_\alpha + 1)$ and $A=\Delta_{\alpha\in \kappa} 
A_\alpha\cap B$. 

Take $\alpha_1, \dots, \alpha_n\in A$, where $n\in \omega$, $n\ge k$, and 
$\alpha_1<\dots< \alpha_n$. We have $\{\alpha_1, \dots, \alpha_k\}=F_\alpha$ for some $\alpha\in 
\kappa$. Since $A\subset B\subset D$, $\beta > \gamma_k(\delta)$ for any $\beta, \delta \in D$ 
such that $\beta >\delta$, and $F_\alpha\in [\alpha_{k}+1]^k$, it follows that 
$\alpha<\alpha_{k+i}$ for $i\ge 1$. Therefore, $\alpha_{k+i}\in A_\alpha \subset A_{F_\alpha}$ 
(because $\alpha_{k+i}\in \Delta_{\beta\in \kappa}A_\beta$). Thus, 
$\{\alpha_1,\dots, \alpha_n\}=F_\alpha \triangle \{\alpha_{k+1},\dots, \alpha_n\}\in F_\alpha+
\langle A_{F_\alpha}\rangle \subset U$. 

Now take $\alpha_1, \dots, \alpha_m\in A$, where $m\in \omega$, $0<m< k$, and 
$\alpha_1<\dots< \alpha_m$, and let $A'$ be any element of $\mathscr U$. We must show that 
$\{\alpha_1, \dots, \alpha_m\}+\langle A'\rangle \cap U\ne \varnothing$. We set $A''=A'\cap 
A\setminus \alpha_m$. For any $\alpha_{m+1},\dots, \alpha_k\in A''$, $\alpha_{m+1}<\dots< \alpha_k$, 
we have $\alpha_1, \dots, \alpha_k\in A$. According to what was shown above, 
$\{\alpha_1, \dots, \alpha_k\}\in U$. On the other hand, $\{\alpha_1, \dots, \alpha_k\}\in 
\{\alpha_1, \dots, \alpha_n\}+\langle A''\rangle \subset \{\alpha_1, \dots, \alpha_n\}+\langle 
A'\rangle $. 

Thus, $\langle A\rangle \subset \overline U$, as required.

It follows from the $\kappa$-completeness of $\mathscr U$ and the definition of the topology of 
$B(\kappa_\mathscr U)$ that the intersection of fewer than $\kappa$ neighborhoods of $\zero$  is 
again a neighborhood of $\zero$. Hence the intersection of fewer that $\kappa$ neighborhoods of any 
point in $B(\kappa_\mathscr U)$ is a neighborhood of this point, and the intersection (union) of 
fewer than $\kappa$ open (closed) sets is open (closed). Thus, any set of cardinality less than 
$\kappa$ in $B(\kappa_\mathscr U)$ is closed (and discrete).
\end{proof}

\end{document}